\documentclass{amsart}
\usepackage{amssymb}
\usepackage{amsmath,amsfonts,amsthm}
\usepackage{tikz}
\usepackage{graphicx}
\usepackage{capt-of}
\usepackage{lipsum}
\graphicspath{ {./images/} }

\newtheorem{theorem}{Theorem}[section]
\newtheorem{claim}{Claim}
\newtheorem{case}{Case}
\newtheorem{lemma}[theorem]{Lemma}

\newtheorem{proposition}[theorem]{Proposition}

\theoremstyle{definition}

\theoremstyle{remark}
\newtheorem{remark}[theorem]{Remark}

\numberwithin{equation}{section}

\newcommand*\xbar[1]{%
	\hbox{%
		\vbox{%
			\hrule height 0.5pt 
			\kern0.5ex
			\hbox{%
				\kern-0.1em
				\ensuremath{#1}%
				\kern-0.1em
			}%
		}%
	}%
}

\newcommand{\g}{\dot{\gamma}}
\newcommand{\p}{\dot{\varphi}}
\newcommand{\f}{\dot{f}}

\begin{document}
	\title{On complete gradient Schouten solitons}

\author{Valter Borges}
\address{Department of Mathematics, 70910-900, Bel\'em-PA, Brazil}
\email{valterborges@mat.ufpa.br}


\date{\today}

\begin{abstract}	
	In this paper, we prove an optimal inequality between the potential function of a complete Schouten soliton and the norm of its gradient. We also prove that these metrics have bounded scalar curvature of defined sign. As an application, we prove that the potential function of a Shrinking Schouten soliton grows linearly and provide optimal estimates for the growth of the volume of geodesic balls.
\end{abstract}

\maketitle

\section{Introduction and Main Results}
We say that a Riemannian manifold $(M^n,g)$ is a {\it gradient Schouten soliton} if there are $f\in C^{\infty}(M)$, called the {\it potential function}, and $\lambda\in\mathbb{R}$ satisfying
\begin{align}\label{fundeq}
	Ric+\nabla\nabla f=\left(\frac{R}{2(n-1)}+\lambda\right)g.
\end{align}
Here $\nabla\nabla f$ is the Hessian of $f$, $R$ is the scalar curvature and $Ric$ is the Ricci tensor of $M$. The soliton is called {\it shrinking}, {\it steady} or {\it expanding}, provided $\lambda$ is positive, zero or negative, respectively. In this case we use the notation $(M^n,g,f,\lambda)$.

More generally, a gradient $\rho$-Einstein manifold, $\rho\in\mathbb{R}$, is a Riemannian manifold $(M,g)$ for which a similar equation holds, namely when
\begin{align}\label{genfundeq}
	Ric+\nabla\nabla f=\left(\rho R+\lambda\right)g
\end{align}
is satisfied for $f\in C^{\infty}(M)$, also called {\it potential function}, and $\lambda\in\mathbb{R}$. With this terminology, Schouten solitons are $1/2(n-1)$-Einstein manifolds. The most notable case occurs when $\rho=0$. This choice leads to Ricci solitons, which are self-similar solutions of the Ricci flow, and the latter played a crucial rule in the proof of the Poincaré conjecture. Examples of $\rho$-Einstein manifolds were given in \cite{catino} by Catino and Mazzieri, where this concept was introduced.

It was shown in \cite{catino1} that for each $\rho$ these metrics are realized as self-similar solutions of a geometric flow known as Ricci-Bourguignon flow (associated to the fixed $\rho$). Such flow was shown to have short time solution for any initial metric whenever $\rho<1/2(n-1)$ and $M$ is compact (see \cite{catino2} and \cite{hamilton1} for more details). On the other hand, if $\rho>1/2(n-1)$ solutions are not likely to exist for any initial metric, even when $M$ is compact \cite{catino2}. It is worth noticing that it is still not known what happens concerning short time existence when $\rho=1/2(n-1)$. This makes this value special.


Concerning examples of Schouten solitons, the simplest ones are Einstein manifolds. Other examples are obtained as follows.	Given $n\geq3$, $k\leq n$ and $\lambda\in\mathbb{R}$, consider an Einstein manifold $(N^{k},g)$ of dimension $k\leq n$ and scalar curvature
	\begin{align}\label{scacurv}
		R_{N}=\frac{2(n-1)k\lambda}{2(n-1)-k}.
	\end{align}
Now, if $(x,p)\in\mathbb{R}^{n-k}\times N^{k}$, $\|x\|^2$ denotes the Euclidean norm, and
	\begin{align}\label{potfuc}
		f(x,p)=\frac{1}{2}\left(\frac{R_{N}}{2(n-1)}+\lambda\right)\|x\|^2	,
	\end{align}
it follows that $(\mathbb{R}^{n-k}\times_{\Gamma} N^k,g,f,\lambda)$ is an $n$ dimensional Schouten soliton, where $g=\left\langle ,\right\rangle+g_{N}$, and $\Gamma$ acts freely on $N$ and by orthogonal transformations on $\mathbb{R}^{n-k}$.


A Schouten soliton is called {\it rigid} if it is isometric to one of those described above \cite{petersen}. In \cite{catino} the following classes of Schouten solitons were proven to be rigid:
\begin{enumerate}
	\item {\it compact};
	\item {\it complete noncompact with $\lambda=0$ and $n\geq3$};
	\item {\it complete noncompact with $\lambda>0$ and $n=3$}.
\end{enumerate}
Motivated by this classification, in this paper we investigate complete noncompact Schouten solitons. Our main result is the following one.

\begin{theorem}\label{maintheorem}
	Let $(M^n,g,f,\lambda)$, $\lambda\neq0$, be a complete noncompact Schouten soliton with $f$ nonconstant. If $\lambda>0$ $($$\lambda<0$, respectively$)$, then the potential function $f$ attains a global minimum $($maximum, respectively$)$ and is unbounded above $($below, respectively$)$. Furthermore,
	\begin{align}\label{positivity}
		0\leq \lambda R\leq2(n-1)\lambda^2,
	\end{align}
	\begin{align}\label{gradestimate}
		2\lambda(f-f_{0})\leq |\nabla f|^2\leq4\lambda(f-f_{0}),
	\end{align}
	with $f_{0}=\displaystyle\min_{p\in M}f(p)$, if $\lambda>0$ $($$f_{0}=\displaystyle\max_{p\in M}f(p)$, if $\lambda<0$, respectively$)$.
\end{theorem}

If the Schouten soliton $(M,g,f,\lambda)$ has constant scalar curvature, then the result above can be proven in the following way. First observe that $(M,g,f,\tilde{\lambda})$ is a Ricci soliton with $\tilde{\lambda}=R/(2(n-1))+\lambda$. To prove $(\ref{positivity})$ in this case, notice that from {\cite[Theorem 1]{ferlo} it follows that $R\in\{0,\tilde{\lambda},\ldots,(n-1)\tilde{\lambda},n\tilde{\lambda}\}$. On the other hand, {\cite[Proposition 3.3]{petersen} rules out the case $R=n\tilde{\lambda}$ for nonconstant $f$. Therefore $0\leq\tilde{\lambda} R\leq(n-1)\tilde{\lambda}^2$. The information on the critical points and $(\ref{gradestimate})$ follow from Hamilton's identity $(\ref{ham_identity})$ below, see {\cite[Lemma 6]{ferlo}.

In order to prove Theorem \ref{maintheorem} when the scalar curvature of $M$ is not necessarily constant we use an ordinary differential inequality satisfied by $|\nabla f|^2$ along suitable curves (see Proposition \ref{strategylemma}). Then we describe the behavior of the solutions of such inequality, obtaining various estimates (see Propositions \ref{C1}, \ref{C2}, \ref{PROP1} and \ref{pp}) which, together with the completeness of $M$, are used to prove Theorem \ref{maintheorem}. We expect the same approach to give a similar result when $\rho\neq1/2(n-1)$.

When $\lambda>0$, the lower bound in $(\ref{positivity})$ was first proved in \cite{catino1} using the corresponding Ricci-Bourguignon flow. They also implicitly showed an upper bound to $R$ for $n=3$. Namely, they proved that $R\leq8\lambda$ on a suitable subset of $M^3$. Our estimate, on the other hand, says that $R\leq4\lambda$, when $n=3$, on the whole manifold. We also point out that a similar version of $(\ref{gradestimate})$ has appeared in \cite{catino,catino1} under additional assumptions for several values of $\rho$, including shrinking Schouten solitons when $n=3$.

In the theory of gradient Ricci solitons, an useful tool is Hamilton's identity. It claims that on a gradient Ricci soliton there exists a constant $c_{0}$ so that
\begin{align}\label{ham_identity}
	R+|\nabla f|^2-2\lambda f=c_{0},
\end{align}
where $f$ is the potential function and $R$ is the scalar curvature. This identity was used in several ways to investigate both geometric and analytical features concerning Ricci solitons \cite{caozhou,pointofview,petersen}. We use Theorem \ref{maintheorem} in an analogous fashion to investigate Schouten solitons. The next results derive from this analogy. They deal with the growth of both the potential function and the volume of geodesic balls of a shrinking Schouten soliton, and are motivated by the results of \cite{caozhou} on Ricci solitons.

\begin{theorem}\label{maintheorem2}
	Let $(M^n,g,f,\lambda)$ be a complete noncompact shrinking Schouten soliton with $f$ nonconstant, $f_{0}=\displaystyle\min_{p\in M}f(p)$ and $q\in M$. Then
	\begin{align}\label{ineq}
		\frac{\lambda}{4}(d(p)-A_{1})^2+f_{0}\leq f(p)\leq\lambda(d(p)+A_{2})^2+f_{0},
	\end{align}
	where $A_{1}$ and $A_{2}$ are positive constants depending on $\lambda$ and the geometry of the soliton on the unity ball $B_{q}(1)$ and $d(p)=d(p,q)>2$.
\end{theorem}

To prove the theorem above we proceed as Cao and Zhou \cite{caozhou}. Namely, we use $(\ref{gradestimate})$ and the second variation of arc length to provide the lower and upper growth estimates $(\ref{ineq})$. It is worth mentioning that the rigid examples mentioned above show that $(\ref{ineq})$ is optimal. They also show that the lower growth estimate in $(\ref{ineq})$ is not true in general for expanding Schouten solitons.

Next we provide volume estimates for geodesic balls on complete Schouten solitons.

\begin{theorem}\label{maintheorem3}
	Let $(M^n,g,f,\lambda)$ be a complete noncompact shrinking Schouten soliton with $f$ nonconstant, $q\in M$, $\delta=\displaystyle\inf_{p\in M}R(p)$ and $\theta=\displaystyle\sup_{p\in M}R(p)$. Then there are positive constants $C_{1}$, $C_{2}$ and $r_{0}$, depending only on $n$ and $\lambda$, so that
	\begin{align}\label{volumestimate}
		C_{1}r^{\frac{n}{2}-\frac{(n-2)\theta}{4(n-1)\lambda}}\leq vol(B_{r}(q))\leq C_{2}r^{n-\frac{(n-2)\delta}{2(n-1)\lambda}},
	\end{align}
	for any $r>r_{0}$. In particular $Vol(M)=+\infty$.
\end{theorem}

By Theorem \ref{maintheorem} we have $\delta,\ \theta\in[0,2(n-1)\lambda]$. Therefore, $(\ref{volumestimate})$ implies that the volume growth of geodesic balls in a complete noncompact shrinking Schouten soliton is at least linear and at most Euclidean. Also, both estimates are optimal, as one can verify the equality for the rigid Schouten solitons described earlier.

This paper is organized as follows. In Section \ref{preliminariesresults}, after recalling some basic facts, we provide the main tools for the development of the paper, which are Proposition \ref{strategylemma} and Lemma \ref{lemmakey}. The former provides our claimed differential inequality for $|\nabla f|^2$, while the latter shows that a certain geometric quantity is monotone along the integral curves of $\nabla f$, which is fundamental in obtaining certain estimates for $|\nabla f|^2$. In Section \ref{gradestimates} these estimates are used to prove Theorem \ref{maintheorem}. In Section \ref{growthvolball}, using Theorem \ref{maintheorem} combined with techniques borrowed from the theory of Ricci solitons \cite{caozhou}, we prove Theorem \ref{maintheorem2} and Theorem \ref{maintheorem3}.

\section{Preliminary Results}\label{preliminariesresults}

We start this section with some identities for Schouten solitons proved in \cite{catino}.

\begin{proposition}[\cite{catino}]If $(M^n,g,f,\lambda)$ is a gradient Schouten soliton, then
	\begin{align}\label{trace}
		\Delta f=n\lambda-\frac{n-2}{2(n-1)}R,
	\end{align}
	\begin{align}\label{Riccizero}
		Ric(\nabla f,X)=0,\ \forall X\in\mathfrak{X}(M),
	\end{align}
	\begin{align}\label{IdentitySch}
		\left\langle\nabla f,\nabla R\right\rangle+\left(\frac{R}{n-1}+2\lambda\right)R=2|Ric|^2.
	\end{align}
\end{proposition}

 Now let $p\in M$ be a regular point of $f$ and $\alpha_{p}:(\omega_{1}(p),\omega_{2}(p))\rightarrow M$ the maximal integral curve of $\frac{\nabla f}{|\nabla f|^2}$ through $p$. When the dependence on $p$ is irrelevant to the context we write $\alpha:(\omega_{1},\omega_{2})\rightarrow M$. In what follows we recall some basic facts about $\alpha$ which are going to be used later:
\begin{enumerate}
	\item Since $(f\circ\alpha)'(s)=1,\ \forall s\in(\omega_{1},\omega_{2})$, for any $[s_{1},s_{2}]\subset(\omega_{1},\omega_{2})$ we have
	\begin{align}\label{lindepen}
		(f\circ\alpha)(s_{2})-(f\circ\alpha)(s_{1})=s_{2}-s_{1},
	\end{align}
	that is, $f\circ\alpha$ is an affine function of $s$.
	\item Fix $s_{0}\in(\omega_{1},\omega_{2})$ and consider the diffeomorphism
	\begin{align}\label{changparam}
		t(s)=\int_{s_{0}}^{s}\frac{d\xi}{|\nabla f(\alpha(\xi))|}, s\in(\omega_{1},\omega_{2}).
	\end{align}
	Observe that for each $s>s_{0}$, $t(s)$ is the length of $\alpha$ restricted to $[s_{0},s]$. Consequently, for any $[s_{1},s_{2}]\subset(\omega_{1},\omega_{2})$ we have
	\begin{align}\label{dist}
		d(\alpha(s_{1}),\alpha(s_{2}))\leq\int_{s_{1}}^{s_{2}}\frac{d\xi}{|\nabla f(\alpha_{p}(\xi))|}.
	\end{align}
	Denote by $s(t)$ the inverse of $t(s)$. Then the curve
	\begin{align}\label{geodintcurv}
		(\alpha\circ s)(t)=\alpha(s(t))
	\end{align}
	is an integral curve of $\frac{\nabla f}{|\nabla f|}$.	Furthermore, a standard computation shows that $(\alpha\circ s)(t)$ is a geodesic for each $t$.
\end{enumerate}

The next result shows that $|\nabla f|^2$ satisfies a certain ordinary differential inequality along $\alpha(s)$. As we will see in the subsequent results, it allows to obtain estimates for $|\nabla f(\alpha(s))|^2$.

\begin{proposition}\label{strategylemma}
	Let $(M^n,g,f,\lambda)$, $\lambda\neq0$, be a Schouten soliton with $f$ nonconstant and $\alpha(s)$, $s\in(\omega_{1},\omega_{2})$, a maximal integral curve of $\frac{\nabla f}{|\nabla f|^2}$. The function $b:(\omega_{1},\omega_{2})\rightarrow\mathbb{R}$, defined by
	\begin{align}\label{composition}
		b(s)=|\nabla f(\alpha(s))|^2,
	\end{align}
	satisfies the differential inequality
	\begin{equation}\label{MainIneq}
		bb''-(b')^2+6\lambda b'-8\lambda^2\geq0,
	\end{equation}
	where $b'$ and $b''$ are the first and the second derivative of $b$ with respect to $s$.
\end{proposition}
\begin{proof}
	Consider the smooth function $a:(\omega_{1},\omega_{2})\rightarrow\mathbb{R}$ given by $a(s)=R(\alpha(s))$. From $d(|\nabla f|^2)(X)=2\nabla\nabla f(X,\nabla f)$ and equation $(\ref{fundeq})$ one has
	\begin{align}\label{eq}
		\begin{split}
			b'(s)&=d(|\nabla f|^2)(\alpha'(s))\\
			&=\left(\frac{R(\alpha(s))}{n-1}+2\lambda\right)df(\alpha'(s))\\
			&=\frac{a(s)}{n-1}+2\lambda,
		\end{split}
	\end{align}
	which after differentiating gives $a'(s)=(n-1)b''(s)$. Consequently,
	\begin{align}\label{int3}
		\begin{split}
			\left\langle\nabla f(\alpha(s)),\nabla R(\alpha(s))\right\rangle&=|\nabla f(\alpha(s))|^2 dR(\alpha'(s))\\
			&=b(s)a'(s)\\
			&=(n-1)b(s)b''(s).
		\end{split}
	\end{align}
	Putting $(\ref{IdentitySch})$, $(\ref{eq})$ and $(\ref{int3})$ together we have
	\begin{align*}
		(n-1)(b(s)b''(s)+b'(s)(b'(s)-2\lambda))&=2|Ric|^2(\alpha(s))\\
		&\geq2(n-1)(b'(s)-2\lambda)^2,
	\end{align*}
	where in the second line we have used the inequality $(n-1)|Ric|^2\geq R^2$, which is a consequence of $(\ref{Riccizero})$, and $(\ref{eq})$ once again. Consequently,
	\begin{equation}\label{theend}
		bb''\geq(b'-2\lambda)(b'-4\lambda),
	\end{equation}
	which finishes the proof.
\end{proof}

If $p\in M$ is a given regular point of $f$, the {\it solution of $(\ref{MainIneq})$ associated to $p$} is the one obtained as in $(\ref{composition})$ using the curve $\alpha_{p}:(\omega_{1}(p),\omega_{2}(p))\rightarrow M$ through $p$.

\begin{remark}\label{simply}
	If $b(s)$, $s\in(s_{1},s_{2})$, is a solution of $(\ref{MainIneq})$ with $\lambda<0$, then $\phi(z)=b(-z)$, $z\in(-s_{2},-s_{1})$, satisfies
	\begin{align*}
		\phi\phi_{zz}-\left(\phi_{z}\right)^2+6\mu\phi_{z}-8\mu^2\geq0,
	\end{align*}
	with $\mu=-\lambda>0$. Here $\phi_{z}$ is the derivative of $\phi$ with respect to $r$.
\end{remark}

\begin{remark}\label{examplater}
	Let $(\mathbb{R}^{n-k}\times_{\Gamma} N^k,g,f,\lambda)$ be the Schouten soliton described in the introduction. It has constant scalar curvature $R=\frac{2(n-1)k\lambda}{2(n-1)-k}$ and, if $k\leq n-1$, its potential function $f$ is not constant and $|\nabla f|^2=\left(\frac{R}{n-1}+2\lambda\right)f$. After a linear change of coordinates using $(\ref{lindepen})$ with the condition $f(\alpha(0))=0$ we can replace $f$ by $s$ obtaining a function $b(s)=\left(\frac{R}{n-1}+2\lambda\right)s$. A simple computation shows that this $b(s)$ is a solution of $(\ref{MainIneq})$.
\end{remark}


Now we investigate a general positive solution of inequality $(\ref{MainIneq})$. In the next section we apply the results obtained here to the solutions associated to Schouten solitons, that is, those constructed in Proposition \ref{strategylemma}. For this purpose, denote by $(\omega_{1},\omega_{2})\subset\mathbb{R}$ any fixed interval, possibly unbounded, and let $b:(\omega_{1},\omega_{2})\rightarrow\mathbb{R}$ be a positive smooth solution of $(\ref{MainIneq})$, with $\lambda\neq0$. Note that the latter condition prevents constant functions as solutions of such inequality. Also, in view of Remark \ref{simply}, we may assume that $\lambda>0$ in all proofs.

We start with the following simple result about positive solutions of $(\ref{MainIneq})$.

\begin{lemma}\label{impRemark2}
	If $b(s),\ s\in(\omega_{1},\omega_{2}),$ is a positive solution of $(\ref{MainIneq})$, then:
	\begin{enumerate}
		\item\label{item1} $b(s)$ has at most one critical point, which must be a minimum. In particular, $b'$ changes sign in $(\omega_{1},\omega_{2})$ at most once;
		\item\label{item2} If there is $s_{1}\in(\omega_{1},\omega_{2})$ so that $b'(s_{1})>\max\{2\lambda,4\lambda\}$ $($$b'(s_{1})<\min\{2\lambda,4\lambda\}$, respectively$)$, then $b'(s)\geq b'(s_{1})$, $\forall s\in[s_{1},\omega_{2})$ $($$b'(s)\leq b'(s_{1})$, $\forall s\in(\omega_{1},s_{1}]$, respectively$)$.
	\end{enumerate}
\end{lemma}
\begin{proof}
	To prove $(\ref{item1})$ assume by contradiction that $s_{1},s_{2}\in(\omega_{1},\omega_{2})$ are two distinct critical points of $b$ with $s_{1}<s_{2}$. Let $i\in\{1,2\}$. Using equation $(\ref{MainIneq})$ we obtain $b(s_{i})b''(s_{i})\geq8\lambda^2>0$, and then $s_{i}$ is a local minimum of $b$, due to $b(s_{i})>0$. Then there exists a local maximum $s_{3}\in(s_{1},s_{2})$, that is $b''(s_{3})\leq0$ and $b'(s_{3})=0$. As before, the latter equality implies $b''(s_{3})>0$, what is a contradiction.
	
	Now we will prove $(\ref{item2})$. If $b'(s_{1})>\max\{2\lambda,4\lambda\}$ then equation $(\ref{MainIneq})$ implies that
	\begin{align}\label{ineqimp}
		b(s_{1})b''(s_{1})\geq(b'(s_{1})-2\lambda)(b'(s_{1})-4\lambda)>0.
	\end{align}
	Since $b(s_{1})>0$, we conclude that $b'$ is increasing around $s_{1}$, which implies that 
	\begin{align}\label{seekingfor}
		b'(s)\geq b'(s_{1})>\max\{2\lambda,4\lambda\},
	\end{align}
	for some $s\geq s_{1}$ close to $s_{1}$. Consequently $(\ref{ineqimp})$ is true for such $s$, which implies that $(\ref{seekingfor})$ holds for $s\in[s_{1},\omega_{2})$. The proof is analogous in the case where $b'(s_{1})<\min\{2\lambda,4\lambda\}$.
\end{proof}

Our first estimate is the following one.

\begin{proposition}\label{C1}
	Let $b:(\omega_{1},\omega_{2})\rightarrow\mathbb{R}$ be a positive smooth solution of $(\ref{MainIneq})$ and $\tilde{s}\in(\omega_{1},\omega_{2})$. If $	c=\min\left\{b(\tilde{s}),b(\tilde{s})e^{\frac{b'(\tilde{s})}{6\lambda}}\right\}>0,$
	then
	\begin{enumerate}
		\item $b(s)\geq c,\ \forall s\in[\tilde{s},\omega_{2})$, if $\lambda>0$;
		\item $b(s)\geq c,\ \forall s\in(\omega_{1},\tilde{s}]$, if $\lambda<0$.
	\end{enumerate}
\end{proposition}
\begin{proof}
	We will prove the proposition in the case where $\lambda>0$. The proof in the remaining case is similar.
	
	Suppose that $b'(s)\leq0,\ \forall s\in[\tilde{s},\omega_{2})$. Note that any solution $b(s)$ of $(\ref{MainIneq})$ satisfies 
	\begin{align*}
		(b'(s)+6\lambda\ln(b(s)))'\geq\frac{(b'(s))^2+8\lambda^2}{b(s)}>0,
	\end{align*}
	what for each $s\in[\tilde{s},\omega_{2})$ implies $e^{b'(\tilde{s})}(b(\tilde{s}))^{6\lambda}\leq e^{b'(s)}(b(s))^{6\lambda}$. As a consequence, for each $s\in[\tilde{s},\omega_{2})$ we have $b(s)\geq b(s)e^{\frac{b'(s)}{6\lambda}}\geq b(\tilde{s})e^{\frac{b'(\tilde{s})}{6\lambda}}\geq c$.
	
	Suppose that $b'(s)\geq0,\ \forall s\in[\tilde{s},\omega_{2})$. Then $b(s)\geq b(\tilde{s})\geq c$.
	
	Suppose that $b'(s)$ changes sign on $[\tilde{s},\omega_{2})$. By Lemma \ref{impRemark2} it happens exactly once at a point of minimum, which we denote by $s_{1}\in(\tilde{s},\omega_{2})$. Then $b'(s)<0$, $\forall s\in[\tilde{s},s_{1})$, $b'(s)>0$, $\forall s\in(s_{1},\omega_{2})$, and $b'(s_{1})=0$. Repeating the argument of the first part we have $b(s)\geq c$, $\forall s\in[\tilde{s},s_{1})$. In particular $b(s_{1})\geq c$. On the other hand, since $b$ is nondecreasing in $[s_{1},\omega_{2})$ we get $b(s)\geq b(s_{1})\geq c$ in this interval. This finishes the proof.
\end{proof}

A solution of $(\ref{MainIneq})$ gives rise to a quantity which turns out to be monotone in many situations, as the lemma below shows.

\begin{lemma}\label{lemmakey}
	Let $b(s)$ be a smooth solution of $(\ref{MainIneq})$ and consider the function
	\begin{align}\label{quantity}
		\sigma(s)=\frac{(b'(s)-4\lambda)^2}{b(s)(b'(s)-2\lambda)},
	\end{align}
	defined whenever $b(s)(b'(s)-2\lambda)\neq0$. If in addition $b'(s)(b'(s)-4\lambda)\neq0$, then
	\begin{align}
		\frac{\sigma'(s)}{b'(s)(b'(s)-4\lambda)}\geq0.
	\end{align}
\end{lemma}
\begin{proof}
	A straightforward computation involving $(\ref{MainIneq})$ gives:
	\begin{align*}
		\frac{\sigma'(s)}{b'(s)(b'(s)-4\lambda)}&=\frac{2(b'-2\lambda)bb''-(b'-4\lambda)(b'(b'-2\lambda)+bb'')}{b'b^2(b'-2\lambda)^2}\\\noalign{\smallskip}
		&=\frac{bb''-(b'-4\lambda)(b'-2\lambda)}{b^2(b'-2\lambda)^2}\\\noalign{\smallskip}
		&\geq0,
	\end{align*}
	for each $s$ such that $b(s)b'(s)(b'(s)-4\lambda)(b'(s)-2\lambda)\neq0$.
\end{proof}

In presence of additional assumptions we can improve Proposition \ref{C1} by using Lemma \ref{lemmakey} in the following way.

\begin{proposition}\label{C2}
	Let $b:(\omega_{1},\omega_{2})\rightarrow\mathbb{R}$ be a positive smooth solution of $(\ref{MainIneq})$. If there is $\tilde{s}\in(\omega_{1},\omega_{2})$ such that $\lambda(b'(\tilde{s})-2\lambda)<0$, then there is a constant $c_{0}>0$ such that $b(s)\geq c_{0},\ \forall s\in(\omega_{1},\omega_{2})$.
\end{proposition}
\begin{proof}
	Because $\lambda(b'(\tilde{s})-2\lambda)<0$ and $b(\tilde{s})>0$ we can consider the positive constant
	\begin{align}
		c_{1}=\min\left\{b(\tilde{s}),-\frac{8\lambda}{\sigma(\tilde{s})}\right\},
	\end{align}
	where $\sigma$ is the function defined in $(\ref{quantity})$. We will show that
	\begin{enumerate}
		\item $b(s)\geq c_{1},\ \forall s\in(\omega_{1},\tilde{s}]$, if $\lambda>0$;
		\item $b(s)\geq c_{1},\ \forall s\in[\tilde{s},\omega_{2})$, if $\lambda<0$.
	\end{enumerate}
	If we take $c_{0}=\min\{c,c_{1}\}$, where $c$ is the constant given by Proposition \ref{C1}, then $b(s)\geq c_{0},\ \forall s\in(\omega_{1},\omega_{2})$, and the proposition is proved. In view of Remark \ref{simply}, we may assume that $\lambda>0$.
	
	Suppose that $b'(s)\leq0$ for $s\in(\omega_{1},\tilde{s}]$. Then $b(s)\geq b(\tilde{s})\geq c_{1},\ \forall s\in(\omega,\tilde{s}]$.
	
	Suppose that $b'(s)\geq0$ for $s\in(\omega_{1},\tilde{s}]$. Consider $\sigma(s),\ s\in(\omega_{1},\tilde{s}]$, defined as in $(\ref{quantity})$. It follows from Lemma \ref{lemmakey} that $\sigma(s)<0$ and $\sigma'(s)\leq0$. Therefore, $\sigma(\tilde{s})\leq\sigma(s)<0$. Equivalently, $0\leq(b'(s))^2\leq(b'(s)-2\lambda)(\sigma(\tilde{s})b(s)+8\lambda)$. On the other hand, item $(\ref{item2})$ of Lemma \ref{impRemark2} implies that $b'(s)<2\lambda<4\lambda$, $\forall s\in(\omega_{1},\tilde{s}]$, which gives $\sigma(\tilde{s})b(s)+8\lambda\leq0$, and then $b(s)\geq-\frac{8\lambda}{\sigma(\tilde{s})}\geq c_{1},\ \forall s\in(\omega,\tilde{s}]$.
	
	Suppose that $b'(s)$ changes sign on $(\omega_{1},\tilde{s}]$. Let $s_{1}\in(\omega_{1},\tilde{s})$ be the only point where $b(s)$ takes its minimum, given by Lemma \ref{impRemark2}. Then $b'(s)<0$, $\forall s\in(\omega_{1},s_{1})$, and $b'(s)\geq0$, $\forall s\in[s_{1},\tilde{s}]$. Proceeding as in the end of the proof of Proposition \ref{C1} we get $b(s)\geq-\frac{8\lambda}{\sigma(\tilde{s})}\geq c_{1}>0,\ \forall s\in(\omega,\tilde{s}]$.
\end{proof}

The next result shows that certain solutions of $(\ref{MainIneq})$ grow exponentially. Later, in Proposition \ref{pp}, we show that these solutions cannot be constructed as those of Proposition \ref{strategylemma}.

\begin{proposition}\label{PROP1}
	Let $b:(\omega_{1},\omega_{2})\rightarrow\mathbb{R}$ be a positive smooth solution of $(\ref{MainIneq})$. If there is $\tilde{s}\in(\omega_{1},\omega_{2})$ satisfying 
	\begin{align}\label{condI}
		(b'(\tilde{s})-2\lambda)(b'(\tilde{s})-4\lambda)>0,
	\end{align}
	and
	\begin{align}\label{condII}
		(b'(\tilde{s})-2\lambda)b'(\tilde{s})>0,
	\end{align}
	then there are positive constants $\kappa$ and $K$ so that
	\begin{enumerate}
		\item \label{ine1POS}
		$b(s)\geq K^{2}e^{2\kappa s},\ \forall s\in[\tilde{s},\omega_{2})$, if $b'(\tilde{s})>\max\{2\lambda,4\lambda\}$;
		\item \label{ine2POS}
		$b(s)\geq K^{2}e^{-2\kappa s},\ \forall s\in(\omega_{1},\tilde{s}]$, if $b'(\tilde{s})<\min\{2\lambda,4\lambda\}$.
	\end{enumerate}
\end{proposition}
\begin{proof}
	In view of Remark \ref{simply} we assume that $\lambda>0$. Inequality $(\ref{condI})$ gives rise to two cases:
	\begin{case}\label{Ca1}
		$b'(\tilde{s})>4\lambda>2\lambda$.
	\end{case}
	It follows from item $(\ref{item2})$ of Lemma \ref{impRemark2} that $b'(s)>4\lambda>2\lambda,\ \forall s\in[\tilde{s},\omega_{2})$. If $\sigma(s)$ is defined as in $(\ref{quantity})$, it follows from Lemma \ref{lemmakey} that $\sigma(s)>0$ and $\sigma'(s)\geq0$, and then $\sigma(s)\geq\tilde{\sigma}=\sigma(\tilde{s})>0,\ \forall s\in[\tilde{s},\omega_{2})$. Equivalently, for each $s\in[\tilde{s},\omega_{2})$ we have
	\begin{align*}
		(b'(s))^2-(\tilde{\sigma}b(s)+8\lambda)b'(s)+2\lambda(\tilde{\sigma}b(s)+8\lambda)\geq0,
	\end{align*}
	which gives 
	\begin{align}\label{signP}
		B_{+}(s)B_{-}(s)\geq0,
	\end{align}
	with
	\begin{align}\label{defB}
		\begin{split}
			B_{+}(s)=b'(s)-\frac{1}{2}\left(\tilde{\sigma}b(s)+8\lambda+\sqrt{\tilde{\sigma}b(s)(\tilde{\sigma}b(s)+8\lambda)}\right),\\
			B_{-}(s)=b'(s)-\frac{1}{2}\left(\tilde{\sigma}b(s)+8\lambda-\sqrt{\tilde{\sigma}b(s)(\tilde{\sigma}b(s)+8\lambda)}\right).
		\end{split}
	\end{align}
	Once $b'(s)>4\lambda$, we cannot have $B_{-}(s)<0$, which by $(\ref{signP})$ implies that $B_{+}(s)\geq0$, and then
	\begin{align*}
		b'(s) \geq\frac{1}{2}(\tilde{\sigma}b(s)+8\lambda),\ \forall s\in[\tilde{s},\omega_{2}).
	\end{align*}
	Integrating the inequality above after dropping its last term we get $b(s)\geq K^2e^{2\kappa s},\ \forall s\in[\tilde{s},\omega_{2})$, with $K^2=b(\tilde{s})e^{-2\kappa\tilde{s}}>0$ and $\kappa=\frac{\sigma(\tilde{s})}{4}>0$, which is $(\ref{ine1POS})$.
	
	\begin{case}\label{Ca2}
		$b'(\tilde{s})<2\lambda<4\lambda$.
	\end{case}
	It follows from $(\ref{condII})$ that $b'(\tilde{s})<0$ and then item $(\ref{item1})$ of Lemma \ref{impRemark2} ensures that $b'(s)<0$, $\forall s\in(\omega_{1},\tilde{s}]$. Consider $\phi(z)=b(-z),\ z\in[\tilde{z},-\omega_{1})$, with $\tilde{z}=-\tilde{s}$ and $\mu=-\lambda$. Then $\phi(z)$ satisfies $(\ref{MainIneq})$ (see Remark \ref{simply}) and $\phi_{z}(z)>0$. Consider $\sigma(z)$ defined as in $(\ref{quantity})$ using $\phi(z)$ and $\mu$. It follows from Lemma \ref{lemmakey} that $\sigma(z)>0$ and $\sigma_{z}(z)\geq0$, $\forall z\in[\tilde{z},-\omega_{1})$. Considering $\tilde{\sigma}=\sigma(\tilde{z})>0$ we have
	\begin{align}\label{ineqality2}
		\tilde{\sigma}\phi(z)+8\mu>0,\ \forall z\in[\tilde{z},-\omega_{1}).
	\end{align}
	To see this first notice that by the definition of $\tilde{\sigma}$ we have $(\tilde{\sigma}\phi(\tilde{z})+8\mu)(\phi_{z}(\tilde{z})-2\mu)=(\phi_{z}(\tilde{z}))^2$. This proves $(\ref{ineqality2})$ at $\tilde{z}$. Using that $\phi$ is increasing in $[\tilde{z},-\omega_{1})$, $(\ref{ineqality2})$ is proved in this interval. Considering functions as in $(\ref{defB})$ and proceeding as before, we get $\phi_{z}(z)\geq\tilde{\sigma}\phi(z)+8\mu,\ \forall z\in[\tilde{z},-\omega_{1})$. Returning to the variable $s$ we get
	\begin{align*}
		b'(s)\leq-\tilde{\sigma}b(s)+8\lambda,\ \forall s\in(\omega_{1},\tilde{s}],
	\end{align*}
	which after integrating gives $b(s)\geq K^2e^{-2\kappa s},\ \forall s\in(\omega_{1},\tilde{s}]$, with $\kappa=\frac{\tilde{\sigma}}{2}>0$ and $K^2=\left(b(\tilde{s})-\frac{8\lambda}{\tilde{\sigma}}\right)e^{\tilde{\sigma}\tilde{s}}=\left(\phi(\tilde{z})+\frac{8\mu}{\tilde{\sigma}}\right)e^{-\tilde{\sigma}\tilde{z}}>0$, by $(\ref{ineqality2})$. This proves $(\ref{ine2POS})$.
\end{proof}

\section{Behavior of $f$ and Proof of Theorem \ref{maintheorem}}\label{gradestimates}

In this section, we use the estimates obtained in the previous one to prove Theorem \ref{maintheorem}. Our first step is to use Proposition \ref{C1} to show that $f$ is unbounded and has at most one critical level.

\begin{proposition}\label{P1}
	Let $(M^n,g,f,\lambda)$, $\lambda\neq0$, be a complete Schouten soliton with $f$ nonconstant. If $p_{0}\in M$ is a critical point of $f$, then
	\begin{enumerate}
		\item \label{one1} $\displaystyle f(p_{0})=\inf_{M}{f}$, if $\lambda>0$;
		\item \label{two2} $\displaystyle f(p_{0})=\sup_{M}{f}$, if $\lambda<0$.
	\end{enumerate}
\end{proposition}
\begin{proof}
	Assume that $\lambda>0$ and suppose that $p_{0}$ is not a global minimum of $f$. As the set of regular values of $f$ is dense in $f(M)$, there is a regular value $f_{1}$ of $f$ such that $f_{1}<f(p_{0})$. Since $f^{-1}(f_{1})$ is a closed submanifold of $M$, there are $p_{1}\in f^{-1}(f_{1})$ and a normalized geodesic $\gamma:[0,l]\rightarrow M$ so that $\gamma(0)=p_{1}$ and $\gamma(l)=p_{0}$. Since $\gamma$ is orthogonal to $f^{-1}(f_{1})$ at $p_{1}$, the uniqueness property of geodesics asserts that $\gamma(t)=(\alpha\circ s)(t)$ as long as $\gamma$ does not hit a critical point of $f$, where $\alpha(s)$ and $s(t)$ are given as in $(\ref{changparam})$ and $(\ref{geodintcurv})$, respectively.
	
	Let $t_{*}\in(0,l]$ be the first parameter satisfying $\nabla f(\gamma(t_{*}))=0$. Using $(\ref{changparam})$ we write $\alpha(s)=\gamma(t(s)),\ s\in[\tilde{s},s_{*})$, where $[\tilde{s},s_{*})$ is the image of $[0,t_{*})$ by $t(s)$ and $\alpha(\tilde{s})=p_{1}$. By Proposition \ref{strategylemma}, the function $b(s)=|\nabla f(\alpha(s))|^2$, $s\in[s_{1},s_{2})$, is a solution of $(\ref{MainIneq})$ and then, by Proposition \ref{C1}, there is a constant $c>0$ so that $b(s)\geq c$, for all $s\in[\tilde{s},s_{*})$. Therefore, $|\nabla f(\gamma(t))|^2\geq c,\ \forall t\in[0,t_{*})$,	what is a contradiction, in view of $|\nabla f(\gamma(t_{*}))|^2=0$. This proves $(\ref{one1})$.
	
	If $\lambda<0$ we use Remark $(\ref{simply})$, reducing the analysis to case $(\ref{one1})$, from which we conclude that $\displaystyle (-f)(p_{0})=\min_{M}{(-f)}$, which finishes the proof of $(\ref{two2})$.
\end{proof}

\begin{proposition}\label{P2}
	Let $(M^n,g,f,\lambda)$, $\lambda\neq0$, be a complete Schouten soliton and $p\in M$ a regular point of $f$. If $b:(\omega_{1}(p),\omega_{2}(p))\rightarrow\mathbb{R}$ is defined as in $(\ref{composition})$, then
	\begin{enumerate}
		\item\label{one} $\omega_{2}(p)=+\infty$ and $\displaystyle\sup_{M}f=+\infty$, if $\lambda>0$;
		\item\label{two} $\omega_{1}(p)=-\infty$ and $\displaystyle\inf_{M}f=-\infty$, if $\lambda<0$.
	\end{enumerate}
\end{proposition}
\begin{proof}
	Assume that $\lambda>0$, otherwise we use Remark $(\ref{simply})$. Assume by contradiction that there is a regular point $p\in M$ such that $\omega_{2}(p)<+\infty$. Let $\alpha_{p}:(\omega_{1}(p),\omega_{2}(p))\rightarrow M$ be the maximal integral curve of $\nabla f/|\nabla f|^2$ so that $\alpha_{p}(\tilde{s})=p$, for some $\tilde{s}\in(\omega_{1}(p),\omega_{2}(p))$. If $b:[\tilde{s},\omega_{2}(p))\rightarrow\mathbb{R}$ is defined as in $(\ref{composition})$ using $\alpha_{p}$, Lemma \ref{C1} gives a constant $c>0$ so that
	\begin{align}\label{estimate}
		|\nabla f(\alpha_{p}(s))|=\sqrt{b(s)}\geq\sqrt{c},
	\end{align}
	for $s\in[\tilde{s},\omega_{2}(p))$. Let $(s_{k})_{k}$ be a monotone sequence in $[\tilde{s},\omega_{2}(p))$ converging to $\omega_{2}(p)$. By $(\ref{dist})$ we have
	\begin{align*}
		d(\alpha_{p}(s_{k}),\alpha_{p}(s_{l}))\leq\frac{1}{\sqrt{c}}|s_{k}-s_{l}|,
	\end{align*}
	and then $(\alpha_{p}(s_{k}))_{k}$ converges to a point $p_{0}\in M$. By continuity, $(\ref{estimate})$ implies that $|\nabla f(p_{0})|\geq\sqrt{c}$, and then $p_{0}$ is a regular point. This is a contradiction, since the integral curve $\alpha_{p_{0}}$ of $\nabla f/|\nabla f|^2$ through $p_{0}$ extends $\alpha_{p}$ past $\omega_{2}(p)$, contradicting the maximality of $\omega_{2}(p)$. Then $\omega_{2}(p)=+\infty$, for any regular point $p\in M$ of $f$.
	
	To see that $f$ is unbounded above, use $(\ref{lindepen})$ and the first part of this proof.
\end{proof}

The next proposition gives a necessary condition for a solution of $(\ref{MainIneq})$ to be constructed as in $(\ref{composition})$.


\begin{proposition}\label{pp}
	Let $(M^n,g,f,\lambda)$, $\lambda\neq0$, be a complete Schouten soliton with $f$ nonconstant and $b:(\omega_{1},\omega_{2})\rightarrow\mathbb{R}$ defined as in $(\ref{composition})$. Then 
	\begin{align}\label{goalTRANSC}
		(b'(s)-2\lambda)(b'(s)-4\lambda)\leq0,\ \forall s\in(\omega_{1},\omega_{2}).
	\end{align}
\end{proposition}

\begin{proof}
	Suppose by contradiction that $(\ref{goalTRANSC})$ is not true and let $\tilde{s}\in(\omega_{1},\omega_{2})$ satisfying $(\ref{condI})$. In what follows we consider $\lambda>0$, otherwise use Remark \ref{simply}.
	
	Assume that $b'(\tilde{s})>4\lambda>2\lambda$. Let $\gamma(t)$ be the reparametrization of $\alpha(s)$ given by $(\ref{geodintcurv})$. Proceeding as in the proof of Proposition \ref{P2} we can see that this curve is well defined in $[\tilde{t},+\infty)$, where $\gamma(\tilde{t})=\alpha(\tilde{s})$. Denote $f(\gamma(t))$ by $f(t)$, $t\in[\tilde{t},+\infty)$, and note that its derivative with respect to $t$ is $\f(t)=|\nabla f(\gamma(t))|$. Using Proposition \ref{PROP1}, $(\ref{ine1POS})$, we can find positive constants $K_{0}$ and $\kappa_{0}$ so that
	\begin{align*}
		\f(t)\geq K_{0}e^{\kappa_{0}f(t)},\ t\in[\tilde{t},+\infty),
	\end{align*}
	where we have used that $s=f(t)+\tilde{s}-f(\tilde{t})$. Integrating the above inequality from $0$ to $t$ we have
	\begin{align*}
		t-\tilde{t}\leq\frac{1}{\kappa_{0}K_{0}}(e^{-\kappa_{0}f(\tilde{t})}-e^{-\kappa_{0}f(t)})\leq\frac{e^{-\kappa_{0}f(\tilde{t})}}{\kappa_{0}K_{0}},
	\end{align*}
	what gives a contradiction when $t\rightarrow+\infty$.
	
	Assume that $b'(\tilde{s})<2\lambda<4\lambda$. We will first prove the following claim.
	\begin{claim}\label{claim}
		$(\omega_{1},\omega_{2})=\mathbb{R}$ and there is $s_{1}<\tilde{s}$ satisfying $(\ref{condI})$ and $(\ref{condII})$.
	\end{claim}
	\begin{proof}
		Notice that according to Proposition \ref{C2} there is a positive constant $c_{0}$ so that $b(s)\geq c_{0},\ \forall s\in(\omega_{1},\omega_{2})$. Arguing as in the proof of Proposition \ref{P2}, this implies that $(\omega_{1},\omega_{2})=\mathbb{R}$. If we show the existence of $s_{1}<\tilde{s}$ so that $b'(s_{1})<0$, then the proof of the second part is finished. Assume by contradiction that $b$ is nondecreasing on $(-\infty,\tilde{s}]$. Then $(\ref{theend})$ implies that $b''(s)\geq\tilde{a},\ \forall s\in(-\infty,\tilde{s}]$, where
		\begin{equation*}
			\tilde{a}=\frac{(2\lambda-b'(\tilde{s}))(4\lambda-b'(\tilde{s}))}{b(\tilde{s})}>0,
		\end{equation*}
		once $b'(s)<2\lambda$ and $b''(s)>0$. But then $0\leq b'(s)\leq\tilde{a}s+b'(\tilde{s})-\tilde{a}\tilde{s}$, which provides a contradiction when $s\rightarrow-\infty$.
	\end{proof}
	Let $s_{1}$ be given as in Claim \ref{claim}. As $\omega_{2}=-\infty$, it follows that $\gamma(t)$ and $f(t)$, defined as before, are defined for $t\in(-\infty,t_{1}]$ where $\gamma(t_{1})=\alpha(s_{1})$. Using Proposition \ref{PROP1}, item $(\ref{ine2POS})$, we have
	\begin{align*}
		\f(t)\geq K_{0}e^{-\kappa_{0}f(t)},\ t\in(-\infty,t_{1}],
	\end{align*}
	for some constants $K_{0}$ and $\kappa_{0}$. Integrating the above inequality from $t$ to $t_{1}$ implies
	\begin{align*}
		t_{1}-t\leq\frac{1}{\kappa_{0}K_{0}}(e^{\kappa_{0}f(t_{1})}-e^{\kappa_{0}f(t)})\leq\frac{e^{\kappa_{0}f(t_{1})}}{\kappa_{0}K_{0}},
	\end{align*}
	what gives a contradiction when $t\rightarrow-\infty$.	
\end{proof}

\begin{remark}\label{optimal}
	Inequality $(\ref{goalTRANSC})$ is optimal. To see that consider the function $b(s)$ of Remark \ref{examplater}. Then $b'(s)=\frac{4(n-1)\lambda}{2(n-1)-k}$, with $k\in\{0,1,\cdots,n-1\}$, which clearly satisfy $(\ref{goalTRANSC})$, with equality if $k=0$ or $k=n-1$.
\end{remark}
We are ready to present the proof of Theorem \ref{maintheorem}.

\begin{proof}[{\bf Proof of Theorem \ref{maintheorem}}]
	
	Assume that $\lambda>0$, otherwise we use Remark \ref{simply} to go back to this case.
	\begin{itemize}
		\item {\bf Proof of (\ref{positivity}).}
		
		Assume that $p$ is a regular point of $f$. Let $b(s)$ be the function defined as $(\ref{composition})$ and $s_{0}\in(\omega_{1}(p),\omega_{2}(p))$ with $\alpha(s_{0})=p$. Then $b'(s_{0})=\displaystyle\frac{R(p)}{n-1}+2\lambda$, and using $(\ref{goalTRANSC})$ we have
		\begin{align}\label{step}
			0\leq R(p)\leq2(n-1)\lambda,
		\end{align}
		as desired.	Now assume that $p$ is a critical point of $f$. If we can approximate $p$ with a sequence of regular points of $f$, the result follows by continuity. Then we will show that this is always possible. In order to do that, assume by contradiction that the set $\mathcal{C}$ of critical points of $f$ has nonempty interior. Let $q$ be a point in the boundary of the interior of $\mathcal{C}$. By approximating $q$ by a sequence in the interior of $\mathcal{C}$ we get
		\begin{align}\label{stepp}
			R(q)=\frac{2(n-1)n\lambda}{n-2},
		\end{align}
		where we have used $(\ref{trace})$ and the fact that $\Delta f$ vanishes in the interior of $\mathcal{C}$. On the other hand, by approximating $q$ by a sequence lying outside $\mathcal{C}$ and using $(\ref{step})$ we get
		\begin{align}\label{steppp}
			R(q)\leq2(n-1)\lambda.
		\end{align}
		If we put $(\ref{stepp})$ and $(\ref{steppp})$ together we get a contradiction.
		\item {\bf Proof of $f(M)=[f_{0},+\infty)$.}
		
		By Proposition \ref{P1} and Proposition \ref{P2} we only need to show that $f$ has a critical point. Let $p\in M$ be a regular point of $f$ and $b:(\omega_{1},+\infty)\rightarrow\mathbb{R}$ the corresponding solution of $(\ref{MainIneq})$. By Proposition \ref{pp} we have $2\lambda\leq b'(s)\leq4\lambda,\ \forall s\in(\omega_{1},+\infty)$, which gives
		\begin{align}\label{goal}
			2\lambda(s_{2}-s_{1})\leq b(s_{2})-b(s_{1})\leq4\lambda(s_{2}-s_{1}),
		\end{align}
		with $s_{1}\leq s_{2}$ in $(\omega_{1},+\infty)$. We claim that $\omega_{1}>-\infty$. Suppose by contradiction that $\omega_{1}=-\infty$. Using $(\ref{goal})$ we have for $s\in(-\infty,s_{2}]$
		\begin{align*}
			b(s)\leq2\lambda s+b(s_{2})-2\lambda s_{2},
		\end{align*}
		showing that $b(s)$ must become negative somewhere below $s_{2}$, which is a contradiction.
		
		Now let $(s_{k})_{k}$ be a sequence so that $\displaystyle\lim_{k\rightarrow+\infty}s_{k}=\omega_{1}$. It follows from $(\ref{goal})$ that $(b(s_{k}))_{k}$ is a Cauchy sequence, and hence it converges to $b_{0}\in[0,+\infty)$. It is not hard to see that $b_{0}$ does not depend on the choice of the sequence $(s_{k})_{k}$. Letting $s_{1}\rightarrow\omega_{1}$ in $(\ref{goal})$ we obtain
		\begin{align}\label{coolineq}
			2\lambda(s-\omega_{1})+b_{0}\leq b(s)\leq4\lambda(s-\omega_{1})+b_{0}.
		\end{align}
		On the other hand, if $s_{l}<s_{k}$ it follows from $(\ref{dist})$ and $(\ref{coolineq})$, that
		\begin{align*}
			\begin{split}
				d(\alpha(s_{k}),\alpha(s_{l}))&\leq\displaystyle\int_{s_{l}}^{s_{k}}\frac{ds}{\sqrt{2\lambda(s-\omega_{1})+b_{0}}}\\\noalign{\smallskip}
				&=\sqrt{\frac{2}{\lambda}}\left(\sqrt{s_{k}+C}-\sqrt{s_{l}+C}\right),\
			\end{split}
		\end{align*}
		with $C=\displaystyle\frac{b_{0}}{2\lambda}-\omega_{1}$. This shows that $(\alpha(s_{k}))_{k}$ is a Cauchy sequence and then it converges to a certain $p_{0}\in M$. Again, $p_{0}$ does not depend on the choice of the sequence. Then, $\displaystyle\lim_{s\rightarrow\omega_{1}}|\nabla f(\alpha(s))|^2=|\nabla f(p_{0})|^2=b_{0}$. Now, since $\omega_{1}>-\infty$ and $(\omega_{1},+\infty)$ is the maximal interval of definition of $b(s)$, we must have $|\nabla f(p_{0})|^2=b_{0}=0$, and then $\displaystyle f(p_{0})=f_{0}=\min_{M}{f}$.
		
		\item {\bf Proof of (\ref{gradestimate}).}
		
		Let $p\in M$ be a regular point of $f$, $b:(\omega_{1},+\infty)\rightarrow\mathbb{R}$ the corresponding solution of $(\ref{MainIneq})$ and $p_{0}\in M$ a critical point of $f$ as constructed above. Let $s_{2}\in(\omega_{1},+\infty)$ so that $\alpha(s_{2})=p$. Using $(\ref{lindepen})$ and $b(s)=|\nabla f(\alpha(s))|^2$ in $(\ref{goal})$ and taking $s_{1}\rightarrow\omega_{1}$ we have the desired inequality at $p$, namely
		\begin{align*}
			2\lambda(f(p)-f_{0})\leq |\nabla f(p)|^2\leq4\lambda(f(p)-f_{0}).
		\end{align*}
		
		If $p\in M$ is a critical point of $f$, then $(\ref{gradestimate})$ is trivially satisfied since, as we have seen, $f(p)=f_{0}$ and $\nabla f(p)=0$.
	\end{itemize}
\end{proof}

\section{Growth Estimates}\label{growthvolball}

In this section we show that $(\ref{positivity})$ and $(\ref{gradestimate})$ can be used to investigate Schouten solitons in a similar way $(\ref{ham_identity})$ is used to understand Ricci solitons. This is done by proving Theorem \ref{maintheorem2} and Theorem \ref{maintheorem3}, whose versions to Ricci solitons were proved in \cite{caozhou}.

\begin{proof}[{\bf Proof of Theorem \ref{maintheorem2}}]
	Let $q\in M$ and $d(q)=d(p,q)$. We follow closely the proof of Proposition 1.2 in \cite{caozhou}. It follows from the second inequality of $(\ref{gradestimate})$ that 
	\begin{align}\label{boundgrad}
		\left|\nabla\sqrt{f-f_{0}}\right|\leq\sqrt{|\lambda|}.
	\end{align}
	This implies that $\sqrt{f-f_{0}}$ is Lipschitz, and then
	\begin{align*}
		\displaystyle f(p)\leq|\lambda|\left(d(p)+\frac{f(q)-f_{0}}{\sqrt{|\lambda|}}\right)^2+f_{0}.
	\end{align*}
	
	Assume that $\lambda>0$. Fix $q\in M$, let $p\in M$ be a point so that $d(p,q)=l>2$ and let $\gamma:[0,l]\rightarrow M$ be a minimizing geodesic, with $\gamma(0)=q$ and $\gamma(l)=p$. Denoting $f(\gamma(s))$ by $f(s)$ and its derivative by $\f(s)$ we get
	\begin{align}\label{uniuse}
		\begin{split}
			\f(l-1)-\f(1)&=\displaystyle\int_{1}^{l-1}\nabla\nabla f(\g(s),\g(s))ds\\ \noalign{\smallskip}
			&=\displaystyle\int_{1}^{l-1}\left(\frac{R}{2(n-1)}+\lambda-Ric(\g(s),\g(s))\right)ds\\ \noalign{\smallskip}
			&\geq\lambda(l-2)-\displaystyle\int_{1}^{l-1}Ric(\g(s),\g(s))ds.
		\end{split}
	\end{align}
	
	In order to estimate the last term in the inequality above we use the second variational formula. Since $\gamma:[0,l]\rightarrow M$ is minimizing, for any continuous piecewise regular function $\varphi:[0,l]\rightarrow\mathbb{R}$ satisfying $\varphi(0)=\varphi(l)=0$ one has
	\begin{align}
		\displaystyle\int_{0}^{l}(\varphi(s))^2Ric(\g(s),\g(s))ds\leq(n-1)\int_{0}^{l}|\p(s)|^2ds.
	\end{align}
	Consider the function
	\begin{align}
		\varphi(s)=
		\left\{ 
		\begin{array}{lll}
			s, & \mbox{if} & s\in[0,1],\\ \noalign{\smallskip}
			1, & \mbox{if} & s\in[1,l-1],\\ \noalign{\smallskip}
			l-s, & \mbox{if} & s\in[l-1,l].
		\end{array} 
		\right.
	\end{align}
	Then
	\begin{align*}
		\displaystyle\int_{1}^{l-1}Ric(\g,\g)ds&=\int_{0}^{l}\varphi^2Ric(\g,\g)ds-\int_{0}^{1}\varphi^2Ric(\g,\g)ds-\int_{l-1}^{l}\varphi^2Ric(\g,\g)ds\\ \noalign{\smallskip}
		&\leq(n-1)\int_{0}^{l}|\p|^2ds+\max_{B_{q}(1)}\left|Ric\right|-\int_{l-1}^{l}\varphi^2Ric(\g,\g)ds\\ \noalign{\smallskip}
		&=2(n-1)+\max_{B_{q}(1)}\left|Ric\right|-\int_{l-1}^{l}\varphi^2Ric(\g,\g)ds.
	\end{align*}
	On the other hand, the last term above can be estimated in the following way
	\begin{align*}
		\displaystyle\int_{l-1}^{l}\varphi^2Ric(\g,\g)ds&\geq\frac{\lambda}{3}-\int_{l-1}^{l}(\varphi(s))^2\nabla\nabla f(\g(s),\g(s))ds\\ \noalign{\smallskip}
		&=\frac{\lambda}{3}-(\varphi(l))^2\f(l)+(\varphi(l-1))^2\f(l-1)-2\int_{l-1}^{l}\varphi(s)\f(s)ds\\ \noalign{\smallskip}
		&\geq\frac{\lambda}{3}+\f(l-1)-\max_{l-1\leq s\leq l}|\f(s)|,
	\end{align*}
	which gives
	\begin{align*}
		\displaystyle\int_{1}^{l-1}Ric(\g,\g)ds\leq2(n-1)-\frac{\lambda}{3}-\f(l-1)+\max_{B_{q}(1)}\left|Ric\right|+\max_{l-1\leq s\leq l}|\f(s)|,
	\end{align*}
	and by $(\ref{uniuse})$,
	\begin{align}\label{then}
		\begin{split}
			\max_{l-1\leq s\leq l}|\f(s)|&\geq\lambda l-2(n-1)-\frac{5\lambda}{3}+\f(1)-\max_{B_{q}(1)}\left|Ric\right|.
		\end{split}
	\end{align}
	For each $s\in[0,l]$, using $(\ref{boundgrad})$, we have
	\begin{align*}
		\begin{split}
			\left|\sqrt{f(s)-f_{0}}-\sqrt{f(l)-f_{0}}\right|\leq\sqrt{\lambda}(l-s)\leq\sqrt{\lambda},
		\end{split}
	\end{align*}
	and then $\sqrt{f(s)-f_{0}}\leq\sqrt{f(l)-f_{0}}+\sqrt{\lambda}$, which for $(\ref{gradestimate})$ gives
	\begin{align}\label{then2}
		\begin{split}
			\max_{l-1\leq s\leq l}|\f(s)|\leq\max_{l-1\leq s\leq l}\sqrt{4\lambda(f(s)-f_{0})}\leq\sqrt{4\lambda(f(l)-f_{0})}+2\lambda.
		\end{split}
	\end{align}
	Analogously
	\begin{align}\label{then3}
		\begin{split}
			|\f(1)|\leq\max_{B_{q}(1)}\left|\nabla f\right|\leq\max_{B_{q}(1)}\sqrt{4\lambda(f-f_{0})}+2\lambda.
		\end{split}
	\end{align}
	Putting $(\ref{then})$, $(\ref{then2})$ and $(\ref{then3})$ together we conclude that
	\begin{align*}
		f(p)\geq\frac{\lambda}{4}(d(p,q)-A_{2})^2+f_{0},
	\end{align*}
	where
	\begin{align*}
		A_{2}=\frac{1}{\lambda}\left(2(n-1)+\frac{17\lambda}{3}+\max_{B_{q}(1)}\left\{\sqrt{4\lambda(f-f_{0})}+\left|Ric\right|\right\}\right)
	\end{align*}
\end{proof}

Now we will prove the growth estimate for the volume of geodesic balls of complete shrinking Schouten solitons, announced in Theorem \ref{maintheorem3}.

Fix $q\in M$, consider $r(p)=d(p,q)$ and $\rho(p)=2\sqrt{f(p)-f_{0}}$. According to $(\ref{boundgrad})$ we have
\begin{align}\label{rhoandf}
	\nabla\rho=\frac{1}{\sqrt{f-f_{0}}}\nabla(f-f_{0}),
\end{align}
and if $r$ is large enough $(\ref{ineq})$ imply that
\begin{align}\label{equivalent}
	\sqrt{\lambda}(r(p)-A)\leq\rho(p)\leq2\sqrt{\lambda}(r(p)+A),
\end{align}
where $A=\max\{A_{1},A_{2}\}$. According to Theorem \ref{maintheorem2}, the set $D(r)=\{p\in M;\rho(p)<r\}$ is precompact and then has finite volume, that is
\begin{align}\label{coarea}
	V(r)=\displaystyle\int_{\overline{D(r)}}dV=\int_{0}^{r}\left(\int_{\partial D(s)}\frac{dS}{|\nabla\rho|}\right)ds<+\infty,
\end{align}
where in the second equality we used the Coarea Formula.

In the next result we present estimates for $V(r)$. This is going to be important in the proof of Theorem \ref{maintheorem3}

\begin{proposition}\label{prop1}
	Let $(M,g,f,\lambda)$ be a complete noncompact shrinking Schouten soliton. Consider the real numbers $\theta=\displaystyle\sup_{p\in M}R(p)$, $\delta=\displaystyle\inf_{p\in M}R(p)$ and $r_{1}>0$. Then
	\begin{align}\label{target}
		\displaystyle \left(\frac{V(r_{1})}{r_{1}^{\frac{n}{2}-\frac{(n-2)\theta}{4(n-1)\lambda}}}\right)r^{\frac{n}{2}-\frac{(n-2)\theta}{4(n-1)\lambda}}\leq V(r)\leq\left(\frac{V(r_{1})}{r_{1}^{n-\frac{(n-2)\delta}{2(n-1)\lambda}}}\right)r^{n-\frac{(n-2)\delta}{2(n-1)\lambda}},
	\end{align}
	for any $r>r_{1}$.
\end{proposition}
\begin{proof}
	Integrating $(\ref{trace})$ on the set $D(r)$ and using the Divergence Theorem gives
	\begin{align}\label{traceuseful}
		\int_{\partial D(r)}|\nabla(f-f_{0})|dS= n\lambda V(r)-\frac{n-2}{2(n-1)}\int_{\overline{D(r)}}RdV.
	\end{align}
	On the other hand, $(\ref{coarea})$ gives
	\begin{align*}
		\begin{split}
			\displaystyle \frac{dV}{dr}(r)&=\int_{\partial D(r)}\frac{dS}{|\nabla\rho|}
		\end{split}
	\end{align*}
	Taking $(\ref{rhoandf})$ into account, 
	the first inequality of $(\ref{gradestimate})$ gives
	\begin{equation*}\label{part}
		\begin{split}
			\displaystyle\int_{\partial D(r)}|\nabla(f-f_{0})|dS&\geq2\lambda\int_{\partial D(r)}\frac{(f-f_{0})dS}{|\nabla(f-f_{0})|}=\lambda r\int_{\partial D(r)}\frac{dS}{|\nabla\rho|}=\lambda r\frac{dV}{dr},
		\end{split}
	\end{equation*}
	and the second inequality of $(\ref{gradestimate})$, analogously, gives
	\begin{align*}\label{part1}
		\begin{split}
			\displaystyle\int_{\partial D(r)}|\nabla(f-f_{0})|dS&\leq2\lambda r\frac{dV}{dr}.
		\end{split}
	\end{align*}
	Plugging all this with $(\ref{traceuseful})$ implies
	\begin{align*}
		\lambda r\frac{dV}{dr}\leq n\lambda V-\frac{n-2}{2(n-1)}\int_{\overline{D(r)}}RdV\leq\left(n\lambda-\frac{(n-2)\delta}{2(n-1)}\right)V,
	\end{align*}
	and
	\begin{align*}
		2\lambda r\frac{dV}{dr}\geq n\lambda V-\frac{n-2}{2(n-1)}\int_{\overline{D(r)}}RdV\geq\left(n\lambda-\frac{(n-2)\theta}{2(n-1)}\right)V.
	\end{align*}
	Integrating these inequalities from $r_{1}>0$ to $r>r_{1}$ we get $(\ref{target})$.
\end{proof}

As a consequence of Proposition \ref{prop1} we have

\begin{proof}[{\bf Proof of Theorem \ref{maintheorem3}}]
	It follows from $(\ref{equivalent})$ that $D(\sqrt{\lambda}(r-A))\subset B_{q}(r)\subset D(2\sqrt{\lambda}(r+A))$ for $r$ sufficiently large. Taking $r_{1}>0$ large enough, it follows from Proposition \ref{prop1} that there are positive constants $C_{1}$ and $C_{2}$ so that
	\begin{align*}
		C_{1}r^{\frac{n}{2}-\frac{(n-2)\delta}{4(n-1)\lambda}}\leq V(\sqrt{\lambda}(r-A))\leq Vol(B_{q}(r))\leq V(2\sqrt{\lambda}(r+A))\leq C_{2}r^{n-\frac{n-2}{2(n-1)\lambda}\theta},
	\end{align*}
	for every $r>r_{1}$. In particular, since $\delta\leq2(n-1)\lambda$, we have
	\begin{align*}
		Vol(B_{q}(r))\geq C_{1}r,
	\end{align*}
	for all $r\geq r_{1}$. This completes the proof.
\end{proof}


\begin{thebibliography}{0}
\bibitem{caozhou} H. D. Cao and D. Zhou, \textit{On complete gradient shrinking Ricci solitons}. Journal of Differential Geometry, v. 85, n. 2, p. 175-186 (2010): .

\bibitem{catino2} G. Catino, L. Cremaschi, Z. Djadli, C. Mantegazza and L. Mazzieri, \textit{The Ricci–Bourguignon flow}. Pacific Journal of Mathematics, v. 287, n. 2, p. 337-370, 2017.

\bibitem{catino} G. Catino and L. Mazzieri, \textit{Gradient Einstein solitons}. Nonlinear Analysis 132 (2016): 66-94.

\bibitem{catino1} G. Catino, L. Mazzieri and S. Mongodi, \textit{Rigidity of gradient Einstein shrinkers}. Communications in Contemporary Mathematics 17.06 (2015): 1550046.


\bibitem{pointofview} M. Eminenti, G. La Nave and C. Mantegazza, \textit{Ricci solitons: the equation point of view}. manuscripta mathematica, 127(3), 345-367, 2008.

\bibitem{ferlo} Fernández-López, M., García-Río, E. \textit{On gradient Ricci solitons with constant scalar curvature}. Proceedings of the American Mathematical Society, 144(1), pp.369-378, 2016

\bibitem{hamilton1} R. S. Hamilton, \textit{Three-manifolds with positive Ricci curvature}. Journal of Differential Geometry, v. 17 (1982), n. 2, p. 255-306.

\bibitem{petersen} P. Petersen and W. Wylie, \textit{Rigidity of gradient Ricci solitons}. Pacific journal of mathematics, v. 241 (2009), n. 2, p. 329-345.
\end{thebibliography}
\end{document}